\documentclass[a4paper,10pt]{amsart}

\usepackage{amsmath, amssymb}
\usepackage[all]{xy}
\usepackage{color}
\usepackage{bm}

\newtheorem{thm}{Theorem}[section]
\newtheorem{prop}[thm]{Proposition}
\newtheorem{lem}[thm]{Lemma}
\newtheorem{cor}[thm]{Corollary}

\theoremstyle{remark}
\newtheorem{remark}[thm]{Remark}
\newtheorem{ex}[thm]{Example}

\newcommand{\K}{{\mathbb K}}

\newcommand{\aut}{\operatorname{aut}}
\newcommand{\ev}{{\rm ev}}
\newcommand{\ad}{{\rm ad}}
\newcommand{\pr}{{\rm pr}}
\newcommand{\comp}{{\rm comp}}

\newcommand{\End}{\operatorname{End}}
\newcommand{\Hom}{{\rm Hom}}
\newcommand{\Lp}{{\rm Lp}}
\newcommand{\Diag}{{\rm Diag}}
\newcommand{\Hur}{\text{Hur}}

\numberwithin{equation}{section}

\title{Cartan calculus in string topology}
\author{
Takahito Naito
}
\date{}
\address{
Nippon Institute of Technology,
4-1 Gakuendai, Miyashiro-machi, Minamisaitama-gun, Saitama, 345-8501 Japan
}
\email{naito.takahito@nit.ac.jp}

\begin{document}

\maketitle

\begin{abstract}
In this manuscript, we investigate a Cartan calculus on the homology of free loop spaces which is introduced by Kuribayashi, Wakatsuki, Yamaguchi and the author.
In particular, it is proved that the Cartan calculus can be described by the loop product and bracket in string topology.
Moreover, by using the descriptions, we show that the loop product behaves well with respect to the Hodge decomposition of the homology of free loop spaces.
\end{abstract}

\section{Introduction and Results}\label{sect:introduction}

Throughout of this manuscript, we assume that $M$ is a closed oriented smooth manifold of dimension $m$ and the coefficient of singular (co)homology is a field $\K$ with ${\rm char} \, \K = 0$.
Let $LM={\rm Map}(S^1 ,M)$ be the free loop space of $M$ and $\aut_1 (M)$ the connected component of the mapping space ${\rm Map}(M,M)$ containing the identity map of $M$. Here, we always identify $S^1$ with ${\mathbb R}/{\mathbb Z}$.

The classical Cartan calculus for differential geometry consists of three types of derivations on $\Omega^*(M)$ the de Rham complex of $M$: the Lie derivative $L_X$, the contraction (interior product) $i_X$ with a vector field $X$ on $M$ and the exterior derivative $d$.
The Lie derivative and the contraction induce actions of the space of vector fields on the de Rham complex. Moreover, these derivations satisfy Cartan (magic) formula $L_X = [d, i_X]$ for any vector field $X$, where $[ \ , \ ]$ denotes the commutator bracket.

This structure is formulated by Fiorenza and Kowalzig in \cite{FK2020} as a homotopy Cartan calculus. 
In \cite{KNWY22}, Kuribayashi, Wakatsuki, Yamaguchi and the author investigated homotopy Cartan calculi relating to the free loop spaces.
We gave a structure of homotopy Cartan calculi on the Hochschild chain complex of $\Omega^*(M)$.
Moreover, as a geometric description of the structure, we constructed operators $L$, $e$ from $\pi_* (\aut_1 (M)) \otimes \K$ to $\End (H^*(LM))$.
In this manuscript, we focus on a homologically defined version of the description
\begin{equation}\label{Le_homology}
L , e : \pi_* (\aut_1 (M)) \otimes \K \longrightarrow \End (H_*(LM));
\end{equation}
see Section \ref{sect:Cartan} for more details.

On the other hands, the homology of $LM$ has rich algebraic structures in string topology initiated by Chas and Sullivan.
In \cite{CS}, they defined a Batalin-Vilkovisky algebra structure on the shifted homology ${\mathbb H}_*(LM) := H_{*+m}(LM)$ with respect to a multiplication $\bullet$ called the loop product and the Batalin-Vilkovisky (BV) operator $\Delta$ which is given by the rotation of loops.
In particular, ${\mathbb H}_*(LM)$ is a Gerstenhaber algebra with the loop bracket $\{ \ , \ \}$; see Section \ref{sect:LoopProd} for details about the algebraic structures.

The aim of this manuscript is to investigate a relation between the loop product (bracket) and the operations \eqref{Le_homology}.
In particular, we show that the operations \eqref{Le_homology} can be described by using the loop product and bracket.
In order to describe the result, we recall the morphism 
$
\Gamma_1 : \pi_* (\Omega \aut_1 (M)) \otimes \K  \to {\mathbb H}_*(LM)
$
due to F\'elix and Thomas \cite{FT2004}; see also Section \ref{sect:Gamma1} for the definition.
They proved that $\Gamma_1$ is injective when $M$ is simply-connected.
By using the morphism $\Gamma_1$, we prove the following theorem, which is a main result in this manuscript.
Here, the notations $L_f$ and $e_f$ means the values of $L$ and $e$ at $f\in \pi_* (\aut_1(M))$, respectively.

\begin{thm}\label{mthm}
Let $h\in \pi_n (\Omega \aut_1(M))$ for $n\geq 1$ and $a \in {\mathbb H}_*(LM)$. Then, the loop product $\Gamma_1 (h) \bullet a$ and the loop bracket $\{ \Gamma_1 (h), a \}$ satisfy the identities
\begin{enumerate}
\item $\Gamma_1 (h) \bullet a = (-1)^{n} e_{\partial (h)} (a)$ \hspace{1em} \text{and}
\item $\{ \Gamma_1 (h), a \} = L_{\partial (h)} (a) - (-1)^{n} \Delta \Gamma_1 (h) \bullet a$
\end{enumerate}
in ${\mathbb H}_*(LM)$. Here, $\partial : \pi_n(\Omega \aut_1 (M)) \stackrel{\cong}{\longrightarrow} \pi_{n+1} (\aut_1 (M))$ is the adjoint map.
Moreover, if $M$ is simply-connected, then the following identity holds;
\begin{enumerate}
\item[(3)] $\{ \Gamma_1 (h), a \} = L_{\partial (h)} (a)$.
\end{enumerate}
\end{thm}

A proof of Theorem \ref{mthm} is stated in Section \ref{sect:Proof_mthm}.
The identities in Theorem \ref{mthm} give us some applications with respect to the loop product.
The following corollary follows immediately from Theorem \ref{mthm}(3).

\begin{cor}\label{cor:L_der}
Let $f\in \pi_* (\aut_1 (M))$. If $M$ is simply-connected, then the operator $L_f : {\mathbb H}_* (LM) \to {\mathbb H}_* (LM)$ is a derivation with respect to the loop product.
\end{cor}

We also discuss a behavior of the loop product in the {\it Hodge decomposition} of $H_*(LM)$.
When $M$ is simply-connected, the homology of $LM$ admits a direct sum decomposition $H_*(LM)\cong \bigoplus_i H_*^{(i)}(LM)$ and each summand $H_*^{(i)}(LM)$ is given as a eigenspace; see \cite{Vi91}.
F\'elix and Thomas \cite{FT2008} proved that the loop product $\bullet$ behaves well with respect to the Hodge decomposition in the following sense;
\[
\bullet : {\mathbb H}_*^{(i)}(LM) \otimes {\mathbb H}_*^{(j)}(LM) \longrightarrow {\mathbb H}_*^{(\leq i + j)}(LM).
\]
An equivariant version of the result is given from Berest, Ramadoss and Zhang \cite{BRZ2021} when the manifold $M$ is rationally elliptic.
Note that the eigenspace $H_*^{(i)}(LM)$ can be defined even if $M$ is not simply-connected. 
We show the following behavior of the loop product of non-simply connected manifolds in the Hodge decomposition.

\begin{thm}\label{thm:Lp_HD}
Let $h\in \pi_n (\Omega \aut_1(M))$ for $n\geq 1$ and $a \in {\mathbb H}^{(i)}_*(LM)$. 
Then, the loop product $\Gamma_1 (h) \bullet a$ is contained in ${\mathbb H}^{(i+1)}_*(LM)$, that is, the loop product $\bullet$ induces
\[
\bullet : {\rm Im} \, \Gamma_1 \otimes {\mathbb H}_*^{(i)}(LM) \longrightarrow {\mathbb H}_*^{(i+1)}(LM).
\]
\end{thm}

This manuscript is organized as follows. 
In Section 2, we recall a homotopy theoretic construction of the loop product and the loop bracket.
In Section 3, geometric and algebraic definitions of the Hodge decomposition of $H_*(LM)$ are described.
The definition of the operators \eqref{Le_homology} is introduced in Section 4.
The morphism $\Gamma_1$ due to F\'elix and Thomas is stated in Section 5.
Moreover, some properties about $\Gamma_1$ with respect to the Hodge decomposition are also observed.
Section 6 is devoted to proving Theorem \ref{mthm}, Corollary \ref{cor:L_der} and Theorem \ref{thm:Lp_HD}.
In Section 7, we give some examples of the image of $\Gamma_1$ when $M$ is a sphere.


\section{Loop product and loop bracket}
\label{sect:LoopProd}
In this section, we first introduce a construction of shriek maps in general setting for recalling a homotopy theoretic description of the loop product due to Cohen and Jones \cite{CJ2002}. 
Consider the pullback diagram of connected spaces
  \[
  \xymatrix{
  E_1 \ar[d]_-{p} & E_2 \ar[d]^-{q} \ar[l]_-{j}\\
  N_1 & N_2. \ar[l]_-{i}
  }
  \]
Here $N_i$ is a compact oriented smooth manifold of dimension $n_i$, $p$ is a fibration and $i$ is an embedding.
Observe that $j$ is an embedding as topological spaces.
Consider the associated disk bundle $\pi : D(\nu) \to N_2$ of the normal bundle of $i$ and an embedding $D(\nu) \hookrightarrow N_1$.
We identify $D(\nu)$ with the embedding image in $N_1$ and
simply write $\tilde{D}(\nu) :=p^{-1}(D(\nu))$ and $\partial \tilde{D}(\nu) := p^{-1}(\partial D(\nu))$.
Since $p$ is a fibration, the homotopy lifting property shows that there exists a map $\tilde{\pi}: \tilde{D}(\nu) \to E_2$ such that $q\circ \tilde{\pi}=\pi \circ p|_{\tilde{D}(\nu)}$.

Let $u \in H^{n_1 -n_2}(D(\nu), \partial D(\nu ))$ be the Thom class and denote by $\tilde{u}$ the pullback of the cohomology class $p^*(u)$ in $H^{n_1 - n_2}( \tilde{D}(\nu), \partial \tilde{D}(\nu ))$.
Then the {\it shriek map} of $j$, denoted by $j_!$, is defined as the composite
\[
\xymatrix@C40pt{
j_! : H_*(E_2)
\ar[r]^-{\text{proj}}
&
H_*(E_2, E_2 \setminus j(E_1))
&
H_*(\tilde{D}(\nu), \partial \tilde{D}(\nu))
\ar[l]^-{\cong}_-{\text{excision}}
\ar[ld]_(.6){\cap \tilde{u}}
\\
&
H_* (\tilde{D}(\nu))
\ar[r]^-{\tilde{\pi}_*}
&
H_* (E_1),
}
\]
where $ \cap \tilde{u}$ denotes the cap product with $\tilde{u}$.

Let $LM\times _{M}LM$ denote the subspace of the product $LM\times LM$ consisting of pairs of loops having the same basepoint, that is, there exists the pullback diagram
\[
\xymatrix{
LM \times LM \ar[d]_-{\ev_0 \times \ev_0}  
&
LM \times_M LM 
\ar[d]^-{\ev_0}
\ar[l]_-{j}
\\
M \times M 
& 
M \ar[l]_-{\Diag}
}
\]
in which $j$ is the inclusion, $\ev_0$ is the evaluation map at $0$ and $\Diag$ is the diagonal map.
Let $\comp : LM \times_M LM \to LM$ be the concatenation of loops defined by
\[
\comp (\gamma_1 , \gamma_2)(t)
 = \left\{
\begin{array}{ll}
\gamma_1 (2t) & \left( 0  \leq t \leq \frac{1}{2}  \right) \\
\gamma_2 (2t-1) & \left( \frac{1}{2} \leq t \leq 1 \right)
\end{array}
\right.
\]
for $(\gamma_1 , \gamma_2) \in LM \times_M LM$.
Then the {\it loop product}, denoted by $\Lp$, is defined as
\[
\xymatrix{
\Lp : H_*(LM)^{\otimes 2}
\ar[r]^-{\times}
&
H_*(LM \times LM)
\ar[r]^-{j_!}
&
H_*(LM\times_M LM)
\ar[r]^-{\comp_*}
&
H_*(LM),
}
\]
where $\times$ denotes the cross product.
The loop product $\Lp$ induces a multiplication on the shifted homology ${\mathbb H}_* (LM) := H_{*+m}(LM)$ defined by
\[
a \bullet b 
:= (-1)^{m \parallel a \parallel }\Lp ( a \otimes b)
=(-1)^{m ( |a|+m) }\Lp (a \otimes b)
\]
for $a$, $b \in {\mathbb H}_*(LM)$, where $\parallel  \hspace{-0.2em} a  \hspace{-0.2em} \parallel$ stands for the degree of $a$ in ${\mathbb H}_*(LM)$.
The definition of $\bullet$ implies that the grading shift morphism $s^m : H_*(LM) \to {\mathbb H}_* (LM)$, $s^m (a) = a$ of degree $m$ fits into the commutative diagram
\[
\xymatrix{
H_* (LM) \otimes H_*(LM) \ar[r]^-{\Lp} \ar[d]_-{s^m \otimes s^m} & H_*(LM) \ar[d]^-{s^m}
\\
{\mathbb H}_*(LM) \otimes {\mathbb H}_* (LM) \ar[r]^-{\bullet} & {\mathbb H}_*(LM).
}
\]
It is well-known that the multiplication $\bullet$ is an associative, unital and commutative multiplication, and moreover, the homology class $c_*([M])\in {\mathbb H}_0(LM)$ is the unit with respect to $\bullet$, where $[M]$ is the fundamental class of $M$ and $c:M\to LM$ is a map which assigns an element $x \in M$ a constant loop at $x$.
Especially, we have
\begin{equation}\label{Lp_unit}
\Lp \left( c_*([M]) \otimes a \right) = (-1)^m a
\end{equation}
in the non-shifted homology $H_*(LM)$.

Next, we recall the Batalin-Vilkovisky (BV) operator $\Delta$ on the homology $H_*(LM)$.
Let $ r : S^1 \times LM \to LM$ be a $S^1$-action of $LM$ induced by the rotation of loops.
Explicitly, $r$ is defined by $r (s, \gamma)(t)=\gamma (t+s)$ for $s,t\in S^1$ and $\gamma \in LM$.
Then, we define $\Delta$ as the composite
\[
\xymatrix{
\Delta : H_*(LM)
\ar[r]^-{[S^1]\times }
&
H_*(S^1 \times LM)
\ar[r]^-{r_*}
&
H_*(LM),
}
\]
where $[S^1]$ is the fundamental class of $S^1$.

Chas and Sullivan \cite{CS} showed that the loop product $\bullet$ and the BV operator $\Delta$ turn ${\mathbb H}_*(LM)$ into a BV-algebra; see also \cite{Ta07}, \cite{Vo05}.
In general, from the result due to Getzler \cite{Ge94}, any BV-algebras have a structure of Gerstenhaber algebras.
Precisely, the bracket $\{ \ , \ \}$ on ${\mathbb H}_*(LM)$ defined by
\[
\{ a , b \} 
:= (-1)^{\parallel  a  \parallel} \Delta (a \bullet b) - (-1)^{\parallel  a  \parallel} \Delta (a) \bullet b - a \bullet \Delta (b)
\]
is a Lie bracket which satisfies the Poisson identity
\begin{equation}\label{Poisson}
\{  a, b_1 \bullet b_2 \} = \{ a, b_1 \} \bullet b_2 + (-1)^{\parallel b_1 \parallel (\parallel a \parallel +1)} b_1 \bullet \{ a , b_2 \}.
\end{equation}
This bracket is called the {\it loop bracket}.

\section{Hodge decomposition of the homology of free loop space}\label{sect:HD}

In this section, we recall geometric and algebraic definitions for the Hodge decomposition of $H_*(LM)$ and compare them.
Let $k \geq 2$ be an integer and $p_k : S^1 \to S^1$ the $k$-fold covering given by $p_k (t) = kt$ for $t\in S^1$.
We denote by $\varphi_k : LM \to LM$ the map induced by $p_k$,
and by $H_*^{(i)}(LM)=\{ a \in H_*(LM) \mid \varphi_{k*}(a) = k^i a \}$ the eigenspace of $\varphi_{k*}$ the induced map in homology corresponding to the eigenvalue $k^i$ for an integer $i\geq 0$.
Remark that the definition of $H_*^{(i)}(LM)$ does not depend on the choice of $k$ since ${\rm char} \, \K =0$.

\begin{lem}\label{lem:HD_BV}
The image of $H_*^{(i)}(LM)$ under $\Delta$ is contained in $H_*^{(i-1)}(LM)$, namely,
$
\Delta \left( H_*^{(i)}(LM)  \right) \subset H_*^{(i-1)}(LM).
$
\end{lem}

\begin{proof}
Let $r : S^1 \times LM \to LM$ be the $S^1$-action stated in Section \ref{sect:LoopProd}. By definition, it is easy to check that the following diagram is commutative:
\begin{equation}\label{diag:rotation_power}
\xymatrix{
S^1 \times LM 
\ar[r]^-{p_k \times 1}
\ar[d]_-{1\times \varphi_k}
&
S^1 \times LM
\ar[r]^-{r}
&
LM
\ar[d]^-{\varphi_k}
\\
S^1 \times LM
\ar[rr]^-{r}
&&
LM.
}
\end{equation}

Observe that $p_{k*} [S^1] = k[S^1]$ in $H_1(S^1)$.
For any $a \in H_*^{(i)}(LM)$, the definition of $\Delta$ and a commutativity of the diagram \eqref{diag:rotation_power} yield that
\begin{align*}
k \cdot \varphi_{k*} \left( \Delta a \right)
&= \varphi_{k*} \circ r_* \left( p_{k*}[S^1]\times a \right) \\
&= r_{*} \circ (1\times \varphi_k)_* \left( [S^1]\times a \right) \\
&= r_* ([S^1]\times k^i a)\\
&= k^i \cdot \Delta (a),
\end{align*}
which completes the proof.
\end{proof}

\begin{remark}
In \cite[Theorem 2]{FT2008}, F\'elix and Thomas proved the same assertion of Lemma \ref{lem:HD_BV} by algebraic way when $M$ is simply-connected.
\end{remark}

Next we recall an algebraic definition for the Hodge decomposition which is described by using a Sullivan model for $LM$ due to Vigu\'e-Poirrier and Sullivan \cite{VS76} in rational homotopy theory.
We refer the reader to the book \cite{FHT} for details about notations and terminology from rational homotopy theory.

From now on, we assume that $M$ is simply-connected in this section.
Let $\wedge V = (\wedge V ,d)$ be a minimal Sullivan model for $M$ and ${\mathcal L} = (\wedge V \otimes \wedge \overline{V}, D)$ the Sullivan model for $LM$ described in \cite[\S 15(c)]{FHT}.
Here, $\overline{V}^i = V^{i+1}$ is the suspension of $V$. We denote by $\bar{v} \in \overline{V}$ the element which corresponds to $v\in V$.
Let $s$ be a derivation of degree $-1$ on $\wedge V \otimes \wedge \overline{V}$ defined by $s(v)= \bar{v}$ and $s(\bar{v})=0$. The differential $D$ of ${\mathcal L}$ is the unique extension of $d$ which satisfies the condition $Ds+sd=0$.

From the definition of $D$, we have a direct sum decomposition ${\mathcal L} = \bigoplus_i {\mathcal L}_{(i)}$ of complexes, where ${\mathcal L}_{(i)} = (\wedge V \otimes \wedge^i \overline{V}, D)$.
Applying the homology functor to the decomposition yields
\begin{equation}\label{HD}
H^* ({\mathcal L}) \cong \bigoplus_{i\geq 0} H^* ({\mathcal L}_{(i)}).
\end{equation}

We here consider $H^*_{(i)}(LM) = \{ \alpha  \in H^*(LM) \mid \varphi_k^* (\alpha) = k^i \alpha \}$ the cohomological version of $H_*^{(i)}(LM)$.
The following proposition asserts that $H^* ({\mathcal L}_{(i)})$ is an algebraic construction for $H^*_{(i)}(LM)$.
It is a well known result, however, we provide a proof for the sake of completeness.

\begin{prop}\label{prop:Model_HD}
The homology $H^* ({\mathcal L}_{(i)})$ is isomorphic to $H^*_{(i)}(LM)$.
\end{prop}

\begin{proof}
First consider a morphism ${\mathcal M}_{\varphi_k} : {\mathcal L} \to {\mathcal L}$ defined by ${\mathcal M}_{\varphi_k} (v)=v$ and ${\mathcal M}_{\varphi_k} (\bar{v}) = k\bar{v}$ for $v\in V$. 
It is easy to check that $\alpha \in H^*({\mathcal L})$ belongs to the direct summand $H^*({\mathcal L}_{(i)})$ if and only if ${\mathcal M}_{\varphi_k}(\alpha) = k^i \alpha$ holds.
Moreover, the result \cite[Theorem 3.2]{BFG1991} asserts that ${\mathcal M}_{\varphi_k}$ is a Sullivan representative for $\varphi_k$. 
Therefore,
\[
H^*_{(i)}(LM) \cong \{ \alpha \in H^* ({\mathcal L}) \mid {\mathcal M}_{\varphi_k}^* (\alpha) = k^i \alpha \} = H^*({\mathcal L}_{(i)}),
\]
and the proof is complete.
\end{proof}

In the rest of this section, we compare with the homological definition $H_*^{(i)}(LM)$ and the cohomological definition $H^*_{(i)}(LM)$.
Let $\langle \ , \ \rangle : H^n(LM)\otimes H_n(LM) \to {\mathbb K}$ be the Kronecker pairing.
Since the characteristic of ${\mathbb K}$ is zero, it is a non-degenerate pairing from the universal coefficient theorem; see \cite[Proposition 5.3]{FHT} for example.
Let us consider a paring
\[
\langle \ , \ \rangle_{ij} : H^n_{(i)}(LM)\otimes H_n^{(j)}(LM) \longrightarrow {\mathbb K}
\]
induced by $\langle \ , \ \rangle$. Then we have the following.

\begin{lem}\label{lem:HD_UCT}
The pairing $\langle \ , \ \rangle_{ij}$ is non-degenerate if and only if $i=j$, that is, $H^n_{(i)}(LM) \cong \Hom (H_n^{(i)}(LM), {\mathbb K})$.
\end{lem}

\begin{proof}
For any $\alpha \in H^n_{(i)}(LM)$ and $a\in H_n^{(j)}(LM)$, we have
\[
k^i \langle \alpha , a \rangle_{ij}
= \langle k^i \alpha, a \rangle
= \langle \varphi^*_k ( \alpha ), a \rangle
= \langle   \alpha , \varphi_{k*} (a) \rangle
= \langle   \alpha , k^j a \rangle
= k^j \langle \alpha , a \rangle_{ij}.
\]
Hence $\langle \ , \ \rangle_{ij} = \delta_{ij}\langle \ , \ \rangle$ holds, where $\delta_{ij}$ is the Kronecker delta.
Since $\langle \ , \ \rangle$ is non-degenerate, the assertion is proved from the direct sum decomposition \eqref{HD}.
\end{proof}

\begin{cor}\label{cor:HD_0}
If $M$ is simply-connected, then the evaluation map $\ev_0 : LM \to M$ induces an isomorphism $\ev_{0*} : H_*^{(0)}(LM) \to H_*(M)$.
\end{cor}

\begin{proof}
It is easy to show that the canonical inclusion $\wedge V \hookrightarrow {\mathcal L}$ induces an isomorphism $H^*(\wedge V) \cong H^*({\mathcal L}_{(0)})$.
Since the inclusion is a Sullivan representative for $\ev_0$, the assertion follows from Proposition \ref{prop:Model_HD} and Lemma \ref{lem:HD_UCT}.
\end{proof}

\section{Geometric Cartan calculus on free loop spaces}
\label{sect:Cartan}

In this section, we give a Cartan calculus on $H_*(LM)$ introduced in Section \ref{sect:introduction} as the operator \eqref{Le_homology}.
Note that it is a homologically defined Cartan calculus of the one due to Kuribayashi, Wakatsuki, Yamaguchi and the author in \cite{KNWY22}.

We first consider two homology classes $c_*([S^n])$ and $\eta_n$ in $H_*(LS^n)$ for $n\geq 2$.
The first one is the homology class in $H_{n}(LS^n)$ which is obtained by $[S^n]$ the fundamental class of $S^n$ via the constant loop map $c:S^n \hookrightarrow LS^n$.
The second one is defined by $\eta_n = (\ad_1)_* ([S^{n-1}])$, where $\ad_1 : S^{n-1} \to \Omega S^n \hookrightarrow LS^n$ is the adjoint of the identity $1 : S^{n-1} \wedge S^1 \cong S^n \to S^n$ given by $\ad_1 (u)(t)= u \wedge t$ for $u\in S^{n-1}$ and $t\in S^1$.

\begin{lem}\label{lem:HomologyClass}
The homology class $\eta_n$ has the following properties.
\begin{enumerate}
\item $\eta_n \in H_{n-1}^{(1)}(LS^n)$.
\item $\Delta (\eta_n)=c_*([S^n])$.
\end{enumerate}
\end{lem}

\begin{proof}
Let $\tilde{p}_k : S^{n-1} \to S^{n-1}$ be the composite
\[
\xymatrix{
S^{n-1}
\cong
S^{n-2}\wedge S^1 
\ar[r]^-{1\wedge p_k}
&
S^{n-2}\wedge S^1 
\cong
S^{n-1}
}
\]
and consider the diagram
\[
\xymatrix@C40pt{
S^{n-1}
\ar[r]^-{\ad_1}
\ar[d]_-{\tilde{p}_k}
&
\Omega S^n
\ar[r]^-{\text{inclusion}}
\ar[d]_-{\varphi_k}
&
LS^n
\ar[d]_-{\varphi_k}
\\
S^{n-1} \ar[r]^-{\ad_1}
&
\Omega S^n \ar[r]^-{\text{inclusion}}
&
LS^n,
}
\]
where $p_k$ and $\varphi_k$ are maps stated in Section \ref{sect:HD}.
Obviously, the right-hand side square is commutative.
We can show that the maps contained in the homotopy set $[S^n, S^n]$ corresponding to $\varphi_k \circ \ad_1$ and $\ad_1 \circ \tilde{p}_k$ through the adjoint congruence $[S^{n-1}, \Omega S^n] \cong [S^n, S^n]$ coincide.
It follows that the left-hand side square is homotopy commutative.
Since $\tilde{p}_{k*}([S^{n-1}])=k[S^{n-1}]$ in $H_{n-1}(S^{n-1})$, 
the assertion (1) follows from a homotopy commutativity of the diagram.

By Lemma \ref{lem:HD_BV} and the assertion (1), $\Delta (\eta_n)$ is contained in $H_n^{(0)}(LS^n)$.
The commutative diagram
\[
\xymatrix{
S^1 \times S^{n-1}
\ar[r]^-{1 \times \ad_1}
\ar[d]_-{\text{proj}}
&
S^1 \times LS^n
\ar[r]^-{r}
&
LS^n
\ar[d]^-{\ev_0}
\\
S^1\wedge S^{n-1}
\ar[rr]^-{\cong}
&
&
S^n
}
\]
yields that $\ev_{0*}\circ \Delta (\eta_n) = [S^n]=\ev_{0*}\circ c_*([S^n])$ in $H_n(S^n)$.
Since $S^n$ is simply-connected for $n\geq 2$, $\ev_{0*}:H_*^{(0)}(LS^n) \stackrel{\cong}{\longrightarrow} H_*(S^n)$ is an isomorphism from Corollary \ref{cor:HD_0}, which completes the proof.
\end{proof}

Given $f\in \pi_n (\aut_1 (M))$ which is represented by $f: S^n \to \aut_1 (M)$.
Let $\ad_f : S^n \times M \to M$ be the adjoint of $f$ defined by $\ad_f (u,x)=f(u)(x)$ for $u\in S^n$, $x\in M$ and denote by $L(\ad_f) : LS^n \times LM \to LM$ the induced map between the free loop spaces.
By using the homology classes $c_*([S^n])$ and $\eta_n$, we define morphisms
\[
L, \ e : \pi_* (\aut_1 (M))\otimes {\mathbb K} \longrightarrow \End ( H_*(LM) )
\]
by $L(f)(a) =  L(\ad_f)_* (c_*([S^n]) \times a )$, $e(f)(a) = L(\ad_f)_* (\eta_n \times a )$ for $a \in H_*(LM)$, respectively. Here the notation $\times$ means the cross product.
We will simply write $L_f := L(f)$ and $e_f := e (f)$.

\begin{lem}\label{lem:CartanFormula}
The operators $L_f$, $e_f$ and $\Delta$ satisfy Cartan formula, namely,
\[
L_f = \Delta \circ e_f - (-1)^{n-1}e_f \circ \Delta.
\]
\end{lem}

\begin{proof}
First, it is easily seen that the following diagram is commutative;
\begin{equation}\label{diag:BV_cross}
\xymatrix@C40pt{
H_*(LS^n)\otimes H_*(LM)
\ar[r]^-{\times}
\ar[d]_-{\Delta \otimes 1 + 1\otimes \Delta}
&
H_*(LS^n \times LM)
\ar[r]^-{L(\ad_f)_*}
\ar[d]_-{\Delta}
&
H_*(LM)
\ar[d]_-{\Delta}
\\
H_*(LS^n)\otimes H_*(LM)
\ar[r]^-{\times}
&
H_*(LS^n \times LM)
\ar[r]^-{L(\ad_f)_*}
&
H_*(LM).
}
\end{equation}
The commutativity of \eqref{diag:BV_cross} and Lemma \ref{lem:HomologyClass} (2) show that
\begin{align*}
\Delta \circ e_f ( a ) 
&= \Delta \circ L(\ad_f)_* (\eta_n \times a)
\\
&=  L(\ad_f)_* \left( \Delta (\eta_n) \times a + (-1)^{n-1} \eta_n \times \Delta (a)  \right)
\\
&= L_f (a) + (-1)^{n-1} e_f \circ \Delta (a)
\end{align*}
for $a \in H_*(LM)$.
\end{proof}

\begin{prop}\label{prop:Le_HD}
The operators $L_f$ and $e_f$ induce
\[
L_f : H_*^{(i)}(LM) \longrightarrow H_*^{(i)}(LM) \hspace{1em} \text{and} \hspace{1em} e_f : H_*^{(i)}(LM) \longrightarrow H_*^{(i+1)}(LM).
\]
\end{prop}

\begin{proof}
Naturality of the cross product $\times$ asserts that it induces
\[
\times : H_*^{(i)}(LS^n) \otimes H_*^{(j)}(LM) \longrightarrow H_*^{(i+j)}(LS^n \times LM).
\]
Moreover, $L(\ad_f)_*$ preserves the degree with respect to the Hodge decomposition.
Therefore, the assertion follows from $c_*([S^n]) \in H_*^{(0)}(LS^n)$ and Lemma \ref{lem:HomologyClass} (1).
\end{proof}


\section{The morphism $\Gamma_1$ and the Hodge decomposition}\label{sect:Gamma1}

In this section, we begin with recalling the morphism $\Gamma_1$ due to F\'elix and Thomas \cite{FT2004}.
Let $g:\Omega \aut_1 (M) \times M \to LM$ be a map defined by $g(\gamma , x)(t)=\gamma(t)(x)$ for $\gamma \in \Omega \aut_1(M)$, $x\in M$ and $t\in S^1$.
Then the map $\Gamma_1$ is defined as the composite
\[
  \xymatrix{
    \Gamma_1 : \pi_n (\Omega\aut_1 (M))\otimes \K
    \ar[r]^-{\Hur}
    &
    H_n (\Omega\aut_1 (M))
    \ar[r]^-{ \times [M]}
    &
    H_{n+m} (\Omega \aut_1 (M) \times M)
    \ar[ld]_-{g_*}
    \\
    &
    H_{n+m}(LM )
    \ar[r]^-{s^m}_-{=}
    &
    {\mathbb H}_n(LM),
  }
\]
where $[M]\in H_m (M)$ is the fundamental class and $\Hur$ is the Hurewicz map, that is, $\Hur (h) = h_* ([S^n])$ for $h:S^n \to \Omega \aut_1 (M)$ in $\pi_n (\Omega \aut_1 (M))$.

\begin{lem}\label{lem:Im_Gamma}
The image of $\Gamma_1$ is contained in ${\mathbb H}_*^{(1)}(LM)$.
\end{lem}
\begin{proof}
Let $\varphi'_k : \Omega \aut_1(M) \to \Omega \aut_1 (M)$ be a map induced by $p_k$ stated in Section \ref{sect:HD}.
Then $\varphi_k \circ g = g \circ (1\times \varphi'_k)$ and it follows that the diagram
\[
\xymatrix{
{\mathbb H}_*(LM)
\ar[r]^-{\varphi_{k*}}
&
{\mathbb H}_*(LM)
\\
\pi_* (\Omega \aut_1 (M)) \otimes {\mathbb K}
\ar[r]^-{\varphi'_{k*}}
\ar[u]^-{\Gamma_1}
&
\pi_* (\Omega \aut_1 (M)) \otimes {\mathbb K}
\ar[u]_-{\Gamma_1}
}
\]
commutes. Since $\varphi'_k$ coincides with the composite
\[
\xymatrix@C70pt{
\varphi'_k : \Omega \aut_1 (M) \ar[r]^-{\text{diagonal}}
&
\left( \Omega \aut_1 (M) \right)^{\times k}
\ar[r]^-{\text{concatenation}}
&
\Omega \aut_1 (M),
}
\]
the induced map between homotopy group $\varphi'_{k*}$ satisfies $\varphi'_{k*}(h) = kh$ for any $h \in \pi_* (\Omega \aut_1 (M))$, which completes the proof.
\end{proof}

\begin{prop}\label{prop:comp_Gamma_BV}
If $M$ is simply-connected, then the composite
\[
\Delta \circ \Gamma_1 : \pi_n (\Omega \aut_1 (M))\otimes {\mathbb K} \to {\mathbb H}_{n+1}(LM)
\]
is zero for $n\geq 0$.
\end{prop}

\begin{proof}
From Lemma \ref{lem:HD_BV} and Proposition \ref{lem:Im_Gamma}, the image of $\Delta \circ \Gamma_1$ is contained in ${\mathbb H}^{(0)}_{n+1}(LM)$. 
Moreover, ${\mathbb H}^{(0)}_{n+1}(LM) \cong H_{n+1+m} (M) = \{ 0 \}$ from Corollary \ref{cor:HD_0}.
\end{proof}

\section{Proofs of Theorem \ref{mthm}, Corollary \ref{cor:L_der} and Theorem \ref{thm:Lp_HD}}\label{sect:Proof_mthm}

We first investigate a relation between the loop product and the morphism induced by $g$ in homology stated in Section \ref{sect:Gamma1}.
In this section, we often regard $H_*(M)$ as a vector subspace of $H_*(LM)$ through the morphism induced by the constant loop map $c:M \hookrightarrow LM$ in homology.
Let $g' : \Omega \aut_1 (M) \times LM \to LM\times_M LM$ be a map defined by 
$
g' (\gamma_1 , \gamma_2) = (g(\gamma_1, \gamma_2(0)), \gamma_2)
$
for $\gamma_1 \in \Omega \aut_1 (M)$ and $\gamma_2 \in LM$ and put $\comp' := \comp \circ g'$.

\begin{lem}\label{lem:shriek}
The following diagram is commutative:
\[
\xymatrix{
H_*(LM)\otimes H_*(LM)
\ar[r]^-{\Lp}
&
H_*(LM)
\\
 H_*( \Omega \aut_1 (M))  \otimes H_*(M) \otimes H_*(LM) 
\ar[u]^-{ (g_* \circ \times ) \otimes 1 }
\ar[r]^-{1 \otimes \Lp}
&
H_*(\Omega \aut_1 (M)) \otimes H_*(LM).
\ar[u]_-{\comp'_* \circ \times }
}
\]
\end{lem}

\begin{proof}
It is easy to check a commutativity of the following diagram:
\[
\xymatrix@C0pt@R30pt{
LM \times LM 
\ar@/_30pt/[rdd]_-{\ev_0 \times \ev_0}
&&
LM \times_M LM 
\ar[ll]_-{j}
\ar@/_20pt/[rdd]
&
\\
&
\Omega \aut_1 (M) \times M \times LM
\ar[d]^-{(1\times \ev_0)\circ \pr_{23}}
\ar[lu]_-{g\times 1}
&
&
\Omega \aut_1 (M) \times ( M\times_M LM)
\ar[ll]_-{1\times j}
\ar[lu]_-{g'}
\ar[d]
\\
&
M\times M
&&
M,
\ar[ll]_-{\Diag}
}
\]
where $\pr_{23}$ is the projection on the second and third factors.
Note that the fiber product $M\times_M LM$ is identified with $LM$ by a homeomorphism given by the composite
\[
\xymatrix{
M\times_M LM
\ar[r]^-{c\times 1}
&
LM\times_M LM
\ar[r]^-{\comp}
&
LM.
}
\]
It follows that the following diagram is commutative:
\[
\xymatrix@C20pt{
H_* (LM\times LM)
\ar[r]^-{j_!}
&
H_*(LM \times_M LM)
\ar[r]^-{\comp_*}
&
H_*(LM)
\\
H_*(\Omega \aut_1 (M) \times M \times LM)
\ar[r]^-{(1 \times j)_!}
\ar[u]^-{(g \times 1)_*}
&
H_*(\Omega \aut_1 (M) \times LM)
\ar[u]_-{g'_*}
\ar[ur]_-{\comp'_*}
\\
H_*(\Omega \aut_1 (M) ) \otimes H_* (M \times LM)
\ar[u]^-{\times}
\ar[r]^-{1\otimes j_!}
&
H_*(\Omega \aut_1 (M) ) \otimes H_* (LM).
\ar[u]_-{\times}
&
}
\]
Therefore, the commutativity of the diagram and the definition of the loop product proves the lemma. 
\end{proof}

\begin{proof}[Proof of Theorem \ref{mthm}]
In order to prove the identity $(1)$ in the assertion, it is enough to show that the following diagram is commutative;
\begin{equation}\label{diag:proof_Lp_e}
\xymatrix{
{\mathbb H}_*(LM)\otimes {\mathbb H}_*(LM)
\ar[r]^-{\bullet}
&
{\mathbb H}_*(LM)
\\
H_*(LM)\otimes H_*(LM)
\ar[r]^-{\Lp}
\ar[u]^-{s^m \otimes s^m}
&
H_*(LM)
\ar[u]_-{s^m}
\\
H_*(\Omega \aut_1 (M)) \otimes H_*(LM)
\ar[u]^-{(-1)^m  g_* \circ ( \times [M])  \otimes 1}
\ar[ru]^-{\comp' \circ \times }
&
\\
\pi_* (\Omega \aut_1(M)) \otimes H_*(LM)
\ar[u]^-{\Hur \otimes 1}
\ar[r]_-{\cong}^-{\partial \otimes 1}
&
\pi_* (\aut_1 (M)) \otimes H_*(LM).
\ar[uu]_-{e'}
}
\end{equation}
Here $e'$ is the adjoint of the operator $e$ in Section \ref{sect:Cartan}.
Observe that the composite of the left-hand side vertical arrows coincides with $ \Gamma_1 \otimes s^m$.
We see that the top square is commutative by definition.
From Lemma \ref{lem:shriek} and the formula \eqref{Lp_unit}, the commutativity of the middle triangle in \eqref{diag:proof_Lp_e} is shown. 
Given $a \in H_*(LM)$ and $h \in  \pi_{n-1}(\Omega \aut_1 (M))$ which is represented by a map $h:S^{n-1}\to \Omega \aut_1 (M)$.
Then we have
\begin{equation}\label{comp'_h}
\comp'_* \circ \times \circ (\Hur \otimes 1)( h \otimes a) = \left( \comp' \circ (h \times 1) \right)_* ([S^{n-1}] \times a).
\end{equation}
On the other hand, let $f:= \partial (h) :S^n \cong S^{n-1}\wedge S^1 \to \aut_1 (M)$ be the adjoint of $h$ given by $f(u\wedge t)(x) := h(u)(t)(x)$ for $u\in S^{n-1}$, $t\in S^1$ and $x\in M$.
By the definition of $e'$, we have
\begin{equation}\label{e_h}
e'\circ (\partial \otimes 1 ) ( h \otimes a) = L(\ad_f)_* (\eta_n \times a) = (L(\ad_f)\circ ( \ad_1 \times 1))_* ([S^{n-1}] \times a),
\end{equation}
where $\ad_f : M \times S^n \to M$ is the adjoint of $f$.
Observe that, for $\gamma \in LM$,
\[
L(\ad_f)\circ (\ad_1 \times 1)(u, \gamma )(t) = h(u)(t)(\gamma (t))
\]
and
\[
\comp' \circ (h \times 1)(u, \gamma)(t)
 = \left\{
\begin{array}{ll}
\gamma (2t) & \left( 0  \leq t \leq \frac{1}{2}  \right) \\
h(u)(2t-1)(\gamma (0)) & \left( \frac{1}{2} \leq t \leq 1 \right).
\end{array}
\right.
\]
Now define three homotopies $H_i : S^{n-1} \times LM  \times I \to LM$ for $i=1,2,3$ by
\[
H_1 (u, \gamma ,s) (t) = \left\{
\begin{array}{ll}
\gamma (0) & \left( 0  \leq t \leq \frac{2}{3}s  \right)\\
h(u)\left( \frac{3t-2s}{3-2s} \right) \left( \gamma \left(  \frac{3t-2s}{3-2s} \right) \right) & \left( \frac{2}{3}s \leq t \leq 1 \right),
\end{array}
\right.
\]
\[
H_2 (u, \gamma ,s) (t) = 
\left\{
\begin{array}{ll}
\gamma (3st) & \left( 0  \leq t \leq \frac{1}{3}  \right)\\
h(u)\left( 3st - s \right) \left( \gamma \left( s \right) \right) & \left( \frac{1}{3} \leq t \leq \frac{2}{3} \right)
\\
h(u)\left( 3t+3s - 3st - 2 \right) \left( \gamma \left( 3t+3s - 3st - 2 \right) \right) & \left( \frac{2}{3} \leq t \leq 1 \right) 
\end{array}
\right.
\]
and
\[
H_3 (u, \gamma, s) (t) = 
\left\{
\begin{array}{ll}
\gamma \left(  \frac{6t}{s+2} \right) & \left( 0  \leq t \leq \frac{s+2}{6}  \right) \\
h(u)\left( \frac{6t-s-2}{s+2} \right) \left( \gamma \left( 0 \right) \right) & \left( \frac{s+2}{6} \leq t \leq \frac{s+2}{3} \right)
\\
\gamma (0) & \left( \frac{s+2}{3} \leq t \leq 1 \right) .
\end{array}
\right.
\]
It is easy to check that
$
L(\ad_f)\circ (\ad_1 \times 1) = H_{1}|_{s=0}$, $H_{1}|_{s=1} = H_{2}|_{s=0}$, $H_{2}|_{s=1} = H_{3}|_{s=0}$ and $H_{3}|_{s=1} = \comp' \circ (h \times 1)$, and these imply that $L(\ad_f)\circ (\ad_1 \times 1)$ is homotopic to $\comp' \circ (h \times 1)$.
Therefore, from \eqref{comp'_h} and \eqref{e_h}, we show that the diagram \eqref{diag:proof_Lp_e} is commutative.

The identity (2) follows from (1) and Lemma \ref{lem:CartanFormula}.
Indeed, we have
\begin{align*}
\{ \Gamma_1 (h), a \} 
&= (-1)^{\parallel \Gamma_1 (h) \parallel} \Delta (\Gamma_1 (h) \bullet a) - (-1)^{\parallel \Gamma_1 (h) \parallel} \Delta (\Gamma_1 (h)) \bullet a - \Gamma_1 (h)\bullet \Delta (a)
\\
&= \left( \Delta e_{\partial (h)} + (-1)^{|\partial (h)|}e_{\partial (h)}\Delta \right) (a) - (-1)^{|h|}\Delta (\Gamma_1 (h)) \bullet a
\\
&= L_{\partial (h)}(a) - (-1)^{|h|}\Delta (\Gamma_1 (h)) \bullet a.
\end{align*}
If $M$ is simply-connected, $\Delta \circ \Gamma_1 =0$ from Proposition \ref{prop:comp_Gamma_BV}. Therefore, we see that the identity (3) holds from (2).
\end{proof}

\begin{proof}[Proof of Corollary \ref{cor:L_der}]
Recall that the loop bracket satisfies Poisson identity \eqref{Poisson}.
By virtue of Theorem \ref{mthm}(3), we have
\begin{align*}
L_f ( a \bullet b ) &= \{ \Gamma_1 \circ \partial^{-1}(f), a \bullet b \}
\\
&=  \{ \Gamma_1 \circ \partial^{-1}(f), a  \} \bullet b + (-1)^{\parallel a \parallel (\parallel \Gamma_1 \circ \partial^{-1}(f) \parallel +1 )} a \bullet \{ \Gamma_1 \circ \partial^{-1}(f), b \} 
\\
&= L_f (a) \bullet b + (-1)^{\parallel L_f \parallel \parallel a \parallel } a \bullet L_f (b ).
\end{align*}
These implies that $L_f$ is a derivation with respect to the loop product.
\end{proof}

\begin{proof}[Proof of Theorem \ref{thm:Lp_HD}]
By Theorem \ref{mthm}(1) and Proposition \ref{prop:Le_HD}, we have
\[
\Gamma_1 (h) \bullet a = (-1)^n e_{\partial (h)}(a) \in {\mathbb H}_*^{(i+1)}(LM)
\]
for $h\in \pi_n (\Omega \aut_1 (M))$ and $a\in {\mathbb H}_*^{(i)}(LM)$, which proves the theorem.
\end{proof}


\section{Examples of the image of $\Gamma_1$}

In this section, we discuss about the image of the morphism $\Gamma_1$ when the case where $M$ is a simply-connected sphere $S^n$ for $n=2,3$.
Since $\Gamma_1$ is injective from \cite{FT2004}, it is enough to consider nontrivial elements in the homotopy group $\pi_*(\aut_1 (S^n))$.

\begin{ex}
Let $S^3$ be the $3$-dimensional sphere. We may regard $S^3$ as the unit sphere in the quaternions ${\mathbb H}$.
Consider a multiplication $\mu:S^3 \times S^3 \to S^3$ induced by the multiplication of ${\mathbb H}$.
The adjoint of $\mu$ induces a map $\ad_{\mu} : S^3 \to \aut_1 (S^3)$, and it is a representative of a nonzero element in $\pi_3 (\aut_1 (S^3)) \otimes \K$.
Since $\Gamma_1$ is injective from \cite[Theorem 2]{FT2004}, $\Gamma_1 (\partial^{-1}(\ad_{\mu}))$ is a nonzero homology class in $H_5 (LS^3)$.
Explicitly, by the definition of $\Gamma_1$, the homology class is obtained by the value of a morphism induced by the composite
\begin{equation}\label{ex_S^3}
\xymatrix@C55pt{
S^2 \times S^3
\ar[r]^-{\partial^{-1}(\ad_{\mu})\times 1}
&
\Omega \aut_1 (S^3)\times S^3 
\ar[r]^-{g}
&
LS^3
}
\end{equation}
in homology at the fundamental class of $S^2 \times S^3$.

On the other hand, since $S^3$ is a Lie group, it is well-known that $LS^3$ splits as the product $\Omega S^3 \times S^3$ with a homeomorphism $\psi : \Omega S^3 \times S^3 \to LS^3$ defined by $\psi (\gamma, u)(t)=\mu (\gamma(t), x)$ for $\gamma \in \Omega S^3$, $u\in S^3$ and $t\in S^1$.
We here recall the adjoint $\ad_1 : S^2 \to \Omega S^3$ and the homology class $\eta_3\in H_2 (\Omega S^3)$ stated in Section \ref{sect:Cartan}.
Then, it is easy to check that the composite \eqref{ex_S^3} coincides with
\[
\xymatrix@C40pt{
S^2 \times S^3
\ar[r]^-{\ad_1 \times 1}
&
\Omega S^3 \times S^3
\ar[r]^-{\psi}
&
LS^3
}
\]
and therefore $\Gamma_1 (\partial^{-1}(\ad_{\mu})) = \psi_* (\eta_3 \times [S^3])$ in $H_5 (LS^3)$.
\end{ex}

\begin{ex}
Let $S^2$ be the $2$-dimensional sphere which is regarded as the unit sphere in ${\mathbb R}^3$.
Consider a $S^3$-action $\mu' : S^3 \times S^2 \to S^2$ induced by the conjugate action ${\mathbb H}\times {\mathbb R}^3 \to {\mathbb R}^3$.
Here ${\mathbb R}^3$ is regarded as the subspace of ${\mathbb H}$ consisting of pure quaternions, that is, quaternions with $0$ scalar part.
It is known fact that the restriction map $\mu'|_{S^3} : S^3 \cong S^3 \times \{ * \} \to S^2$ is the Hopf fibration.
Hence the adjoint of $\mu'$ denoted by $\ad_{\mu'} : S^3 \to \aut_1 (S^2)$ is a representative of a nonzero element in $\pi_3 (\aut_1 (S^2)) \otimes \K$. 
Therefore, by the injectivity of $\Gamma_1$, we obtain a nonzero homology class $\Gamma_1 (\partial^{-1}(\ad_{\mu'}))$ which is obtained by the value of a morphism induced by the composite
\[
\xymatrix@C55pt{
S^2 \times S^2
\ar[r]^-{\partial^{-1}(\ad_{\mu'})\times 1}
&
\Omega \aut_1 (S^2) \times S^2 
\ar[r]^-{g}
&
LS^2
}
\]
in homology at the fundamental class of $S^2 \times S^2$.
\end{ex}

\section*{Acknowledgement}

We would like to thank Katsuhiko Kuribayashi, Shun Wakatsuki and Toshihiro Yamaguchi for valuable discussions and suggestions.

\bibliographystyle{alpha}
\bibliography{bibliography}

\end{document}